\documentclass[10pt,a4paper,reqno]{amsart}
\usepackage{palatino}
\usepackage{amsfonts,amsmath}
\usepackage{amsthm}
\usepackage{graphicx,enumerate}
\usepackage[hmarginratio=1:1,hscale=0.75]{geometry}
\usepackage[latin1]{inputenc}
\usepackage[all]{xy}

\usepackage{mathrsfs}

\usepackage{amssymb,mathtools}
\usepackage{stmaryrd}
\usepackage{hyperref}
\usepackage{color}
\usepackage[T1]{fontenc}
\usepackage{lmodern}
\usepackage{soul,color}

\usepackage[normalem]{ulem}
\usepackage{amsmath}
\newcommand{\stkout}[1]{\ifmmode\text{\sout{\ensuremath{#1}}}\else\sout{#1}\fi}

\usepackage{cancel}

\newtheorem{counter}{Counter}
\newtheorem{lem}[counter]{Lemma}
\newtheorem{defn}[counter]{Definition}
\newtheorem{thm}[counter]{Theorem}

\newcommand{\tcr}[1]{\textcolor{black}{#1}}

\setstcolor{red}

\newcommand{\R}{\mathbb{R}}
 
\renewcommand{\lg}{\langle} 
\newcommand{\rg}{\rangle} 

\newcommand{\Lra}{\Longrightarrow}
 
\newcommand{\da}{\downarrow}

\newcommand{\sse}{\subseteq}

\newcommand{\la}{\lambda} 

\newcommand{\pd}{\partial}

\newcommand{\fal}{\forall}
\newcommand{\8}{\infty}

\newcommand{\vph}{\varphi}
\newcommand{\vep}{\varepsilon} 
 
\newcommand{\dt}{\delta}

\newcommand{\om}{\Omega}

\renewcommand{\d}{\,\text{d}}

\newcommand{\magenta}{ \color{black} }

\newcommand{\red}{\color{black}}
\newcommand{\blue}{\color{black}}
\definecolor{mygreen}{rgb}{0.1,0.75,0.2}

\title[Regularity and monotonicity for an epitaxial growth model]{Regularity  and monotonicity for solutions 
	to a continuum model of epitaxial growth with nonlocal elastic effects}

\author{Yuan Gao}
\author{Xin Yang Lu}
\author{Chong Wang}
	
\address{Department of Mathematics, Duke University, Durham NC 27708, USA}
\email{yuangao@math.duke.edu}	

\address{ Department of Mathematical Sciences, Lakehead University, Thunder Bay, Ontario, P7B 5E1, Canada AND Department of Mathematics and Statistics, McGill University,  Montreal, Quebec, H3A 0B9, Canada}
\email{xlu8@lakeheadu.ca}

\address{Department of Mathematics and Statistics, McMaster University, Hamilton, Ontario, L8S 4K1, Canada}
\email{wangc196@mcmaster.ca}

\begin{document}

\begin{abstract}
We study {\color{black}
the following parabolic  nonlocal 4th order degenerate equation
\[u_t  = - \Big[2\pi H(u_x)+\ln(u_{xx}+a)+\frac32(u_{xx}+a)^2 \Big]_{xx},\]
  arising} from the epitaxial growth on crystalline materials. {\color{black} Here $H$ denotes the Hilbert transform, and $a>0$ is a given parameter. By relying on the theory of gradient flows approach, 
we first prove the global existence of a variational inequality solution with {\blue a general initial datum}.
Furthermore, to obtain a global strong solution, the main difficulty is  the singularity of the logarithmic term when $u_{xx} +a $
approaches zero. Thus we  show that,
 if {\blue the initial datum $u_0$ satisfies} $(u_0)_{xx}+a$ is uniformly bounded away from zero
then such property is preserved for all positive times.}
Finally, we will prove several higher {\red regularity  results for this global strong solution.}
 These finer  properties provide rigorous justification for{\magenta the  global-in-time} monotone solution to epitaxial growth model with nonlocal elastic effects on vicinal  surface.
\end{abstract}

\maketitle	



\section{Introduction}
{\magenta  One of the most affordable manufacturing processes to produce several
key semiconductor materials is the epitaxial growth on crystal surface \cite{pimpinelli1998physics, Zang1990}.}
%
%
It is also used  to design{\magenta experimental} materials to show high  
temperature {\magenta superconducting properties}, or the quantum anomalous hall effect, in magnetic topological insulators \cite{ChangXue2013}. During the growth process, different coherent states are formed due to the balance of competing influences, which {\blue is} crucial to {\blue the study of} the various structures of crystal surfaces. {\magenta The presence of these} complicated competing effects usually {\blue leads} to a high-order, nonlinear, nonlocal  model, which requires mathematical validations at both macroscopic and microscopic scales.

The formal derivation of the continuum limit 
{\magenta generally starts}
 from a mesoscopic description, such as Burton-Cabrera-Frank (BCF) step models \cite{BCF, Tersoff1995-Step-bunching}; see \cite{Duport1995b, Tang1997-Flattening, Weinan2001-Continuum, Xiang2002-Derivation, Xiang2004-Misfit}. {\magenta In {\blue these models},
 they  considered} a discrete energy functional
$E_i=\sum_{i,j}c_1\ln|x_i-x_j|+c_2\frac{1}{(x_i-x_j)^2}$
to  incorporate the global elastic interaction between steps $x_i$ and $x_j$, where $c_1,c_2$ are proper scaling constants. 
The resulting epitaxial growth model in terms of  the 
{\red continuum variable $h(x,t)$}, {\magenta which represents} {\red the thin film height,} is
\begin{equation}\label{equation h}
h_t= -[2\pi H(h_x)+(h_x^{-1}+3h_x )h_{xx}  ]_{xx}.
\end{equation}
{\red Here
\[H(v)(x):=\frac{1}{|I|}\int_I v(x-y)\cot\frac{\pi y}{|I|}\d y, \]
denotes the Hilbert transform.}
Under the assumption that {\red the slope $h_x$} of the thin film height
${\red h} $ is 
{\red strictly positive}, i.e., $ h_x >0$ for any $t>0$,
{\red Gao, Liu and Lu \cite{gao2017continuum}} 
{\magenta gave a rigorous proof of  the convergence from mesoscopic BCF step models to \eqref{equation h}.} They also  obtained the local smooth solution 
whose monotonicity {\color{black} is preserved up to a (positive) time.}

{\magenta Concerning} global solutions,  
Dal Maso, Fonseca and Leoni \cite{Leoni2014}, and Fonseca, Leoni and {\red Lu} \cite{Leoni2015}
  showed the existence {\blue of} a global weak solution by considering another equation 
 for
 the anti-derivative {\magenta$u$} of $h$, {\magenta which satisfies}  $h_x=u_{xx}+a$ for some constant $a>0$, under the assumption that the initial datum is sufficiently regular. 
That is, the authors considered the parabolic variational
equation
\begin{equation}
u_t= -[2\pi H(u_x)  +\Phi_a'(u_{xx}) ]_{xx}, \qquad
\Phi_a(\xi):=\Phi(\xi+a),\qquad
\Phi(\xi):= \begin{cases}
+\8,& \xi<0,\\
0,& \xi=0,\\
\xi\ln\xi  +\xi^3/2,& \xi>0,
\end{cases}
\label{equation u}
\end{equation}
on the {\magenta spatial} domain  $I:= {\red (0,1) }$ with 
{\magenta periodic boundary conditions, and}
time $t\ge 0$.

It has been shown in \cite[Section~2]{gao2017continuum} that if {\red $u_{xx}(t)+a>0$} for all 
 $t\geq 0$ then \eqref{equation h} can be formally written
in the form of $L^2$-gradient flow
\begin{equation}\label{u_eq}
u_t = -\frac{\dt E}{\dt u}=  - [2\pi H(u_x)+\ln(u_{xx}+a)+\frac32(u_{xx}+a)^2 ]_{xx},
\end{equation}
where 
\begin{align}
 E(u) &:=\frac{1}{|I|}\int_I \int_I \ln \left|\sin( \pi(x-y) )\right| (u_{xx}+a)(u_{yy}+a)  \d y\d x
+\int_I \Phi(u_{xx}+a)\d x ,\label{definition of E}
\\
\frac{\dt E}{\dt u}&:=[2\pi H(u_x)+\Phi'(u_{xx}+a)]_{xx} .\label{subgradient of E}
\end{align}
However, the two equations \eqref{equation h} and \eqref{u_eq} are equivalent {\red 
	under the assumption $h_x=u_{xx}+a$ is strictly positive for any time;} see \cite[Theorem~3.1]{Leoni2014}
and also \cite[Section~2.5]{gao2017continuum}.
To
the best of our knowledge, 
 for arbitrarily large times, whether the solution to \eqref{equation h} remains strictly monotone 
{\red  is a long-standing question that is never addressed in  previous literature}. We also refer to \cite{GG2010, kohn_giga_2011, liu2016existence, liu2017analytical, xu2018existence, gao2018maximal, gao2019gradient,  gao2019global} and the references therein for some other related 4th order degenerate equations. {\red Instead of the nonlocal term $H(h_x)$ in \eqref{equation h} resulting from the global interactions between mesoscopic steps, the   4th order degenerate equations in \cite{GG2010, kohn_giga_2011, liu2016existence, liu2017analytical, xu2018existence, gao2018maximal, gao2019gradient,  gao2019global}  involve only locally defined terms $h_x, h_{xx}$ due to the nearest-neighbor interactions between steps.}

\bigskip

{\magenta  In this paper, we will study finer regularity properties of solutions to \eqref{equation u}.
First, we will prove the existence of a solution in}
the evolution variational inequality (EVI) {\magenta sense (see Definition \ref{defweak} below)  without the extra regularity assumption
\cite[(5)]{Leoni2015} }
 on the initial {\blue datum}.
 The second goal is to   prove the higher order {\red regularity} and long time behavior of the global strong solution. {\magenta This is mainly achieved} by carefully studying the sub-differential of the total energy $E$. An important consequence is that the solution to \eqref{equation h} remains strictly monotone {\red $h_x>0$ for any large time}, which also gives the justification that the hydrodynamic limit proved in \cite{gao2017continuum} from the mesoscopic step models to \eqref{equation h} is  indeed true {\red for any large time}. Another consequence is that the global strong solution  converges exponentially to its unique equilibrium.

\bigskip

{\magenta One of the key issues is that the logarithmic term $\ln(u_{xx}+a)$  
has an asymptote as $u_{xx}+a$ approaches zero. This also leads to the issue that the 
sub-differential of $E$ is not easy to characterize, and can become quite complicated
as $u_{xx}+a$ approaches zero. To overcome these issues, 
we will exploit the gradient flow structure of \eqref{equation u} to obtain {\blue the important a priori estimate}. The theory of gradient flows
in Hilbert spaces is very well developed.
For the corresponding results in metric spaces, we refer the interested {\blue readers} to
 \cite{AGS}. 
}

\medskip

{\magenta Our main functional spaces will be}
{\red 
	%
$L^2_{per_0}(I)$, i.e. the space of functions that are
square integrable, periodic, with zero average, endowed with the inner product
$$
\lg u,v\rg:=\int_I uv \d x,\notag
$$
{\magenta and} $W^{k,p}_{per_0}(I)$, {\magenta defined as} the space of functions in 
$W^{k,p}(I)$ that are periodic and have zero average.

\begin{defn}\label{defweak}
Given {\blue an initial datum} $u_0 \in  \overline{D(E)}^{\|\cdot\|_{L^{2}(I)}}$, we call $u:[0,+\8)\to  \overline{D(E)}^{\|\cdot\|_{L^{2}(I)}}$ a variational inequality solution to \eqref{u_eq} if $u(t)$ is a locally absolutely continuous curve such that $\lim_{t\to 0} u(t)=u_0 $ in $L^2(I)$ and
 \begin{equation}
\frac12 \frac{\d}{\d t} \|u(t) - v\|_{L^2(I)}^2 \le E(v)-E(u(t)) \quad \text{for a.e. } t>0,\,\forall v\in { D(E)}.
\end{equation}
\end{defn}
Here, and in the rest of the paper, $D(\cdot)$ will denote the effective domain of a
given functional, 
{\magenta i.e.
\[ D(E)=\{ v\in L^2(I): E(v)<+\8 \}, \]
}
  and $ \overline{D(E)}^{\|\cdot\|_{L^{2}(I)}}$
denotes the closure of $D(E)$ with respect to the $L^2$ distance.
}

Let us state the main results below.
\begin{thm}
Let $E$ be the energy defined in \eqref{definition of E}.
	Given {\blue an initial datum} $u_0 \in  \overline{D(E)}^{\|\cdot\|_{L^{2}(I)}}$, 
	then
	equation \eqref{u_eq} admits
 {\red a unique EVI solution $u$, in the sense of Definition \ref{defweak},
satisfying}
		\begin{equation}
		u \in  L^\8_{loc}(   0,+\8 ; W^{2,3}_{per_0}( I))  ,\qquad 
		\label{regularity h part 1}		
		\end{equation}
		Moreover, if $E(u_0)<+\8$, we have   $u_t\in L^\8( 0,+\8 ; L^{2}( I)).$
\label{main theorem 1}
\end{thm}
Note that \eqref{regularity h part 1} allows 
 more general initial {\blue datum} compared {with} both
\cite[Theorem~1]{Leoni2015} and \cite[Theorem~1.1]{gao2017continuum}.

\begin{thm}\label{main theorem 2}
	{\magenta Assume} {\blue the initial datum} {\red 
		$$u_0\in D(\pd E)=\{ v\in L^2(I): \text{the sub-differential } \pd E(v)\neq \emptyset \},  $$ 
		and $(u_{0})_{xx}+a\geq c>0$ for some $c>0$}, then
	the solution given by Theorem \ref{main theorem 1} is a global strong solution to 
\begin{equation}
u_t = - [2\pi H(u_x)+\ln(u_{xx}+a)+\frac32(u_{xx}+a)^2 ]_{xx}
\end{equation}	
	 and  satisfies:
\begin{enumerate}[(i)]
\item The sub-differential {\magenta $\pd E(u( t))$} is single-valued 
{\magenta for all $t$,}
and is given by 
$$\frac{\delta E}{\delta u}:=[2\pi H(u_x {\magenta (t) } )+\Phi'(u_{xx}{\magenta (t) }+a)]_{xx};$$
\item The right metric derivative satisfies
\[ |u_+'|(t):=\lim_{s\da t^+} \frac{\|u(s)-u(t) \|_{L^2(I)} }{s-t}
= \bigg\|\frac{\delta E}{\delta u}\bigg\|_{L^2(I)}\]
for all $t>0$. 

\item The map $t\mapsto E(u(t))$ is convex, while
$t\mapsto\|\frac{\delta E}{\delta u}( t)\|_{L^2(I)}\exp(2(\sqrt{3}-2\ln 2) t)$
 is nonincreasing and right continuous.

\item It holds
\begin{align}
\ u_{xxx},\ [\Phi'(u_{xx}+a)]_{xx} &\in  L^2( 0,+\8 ; L^{2}( I)) \cap L^\8( 0,+\8 ; L^{2}( I))
\label{more regularity 1},\\
[\ln (u_{xx}+a)]_x,\ [ (u_{xx}+a)^2]_x&\in L^2( 0,+\8 ; C^{0}( I)) \cap L^\8( 0,+\8 ; C^{0}( I)),
\label{more regularity 2} \\
u_{xxx}&\in  L^2( 0,+\8 ; C^{0}( I))\cap L^\8(0,+\8 ; C^{0}( I)) \label{more regularity 3},\\
u_{xxxx},\ [\ln (u_{xx}(\cdot,t)+a)  ]_{xx},
\ [(u_{xx}(\cdot,t)+a)^2]_{xx}&\in  \ L^\8(  0,+\8  ; L^{2}( I)); \label{more regularity 4}
\end{align}
\item {\red 
	{\magenta There exists} a lower bound $c^*>0$, defined in \eqref{lowerB49}, such that}
\begin{eqnarray*}
{\red u_{xx}(t)+a\geq c^*>0  \text{ for any } t>0;}
\end{eqnarray*}
\item The exponential decay to the unique equilibrium $u^*\equiv 0$
\begin{equation}
\|u(t)-u^*\|^2_{L^2(I)} \leq \frac{1}{C} (E(u_0)-E(u^*)) e^{-4Ct} \,\,\,\text{ for all }t>0,
\end{equation}
{\red holds}, where $C:=\sqrt{3}-2\ln 2>0.$
\end{enumerate}
	\end{thm}
	{\red We remark the assumptions in Theorem \ref{main theorem 2} on {\blue the initial datum} are  equivalent to
	$$(u_{0})_{xx}+a\geq c>0, \quad  \|\pd E(u_0)\|_{L^2(I)}= 
	\Big\|\frac{\delta E}{\delta u}(u_0)\Big\|_{L^2(I)} <+\8$$
	due to the calculations for  sub-differential $\pd E$ in Lemma \ref{sub-differential single-valued}
	below.}
	
	{\magenta As an important consequence, since the strong solution
		$u$ to \eqref{equation u} satisfies $u_{xx}+a\ge c^*>0$, 
		\eqref{equation h} and \eqref{equation u} are equivalent in a rigorous way,
		and the function $h$, whose slope is $h_x=u_{xx}+a\geq c^*>0$ 
		and {\blue which} satisfies $\int_I h\d x=a$,
		 is effectively a solution to \eqref{equation h}.}

%
  {Another conclusion is that, {\red for a given $a$,} the {\magenta steady
  state solution} to \eqref{equation h} must be an oblique line with slope $a$.

\medskip


{\magenta
This paper is structured as follows. In Section \ref{sec2} we show that $E$ is
$\la$-convex (see Definition \ref{lambda convexity} below) and lower semi-continuous in $L^2(I)$. This  allows us to use the theory
of gradient flows of $\la$-convex energies from \cite{AGS}, to prove Theorem \ref{main theorem 1}.
In Section \ref{sec3}  we  perform  {\blue crucial} a priori estimates and calculations of the sub-differential of $E$,
 showing that it is indeed single-valued. This finally allows us to prove the higher regularity
results in Theorem \ref{main theorem 2}.
}

\section{A gradient flow approach for EVI solution}\label{sec2}
In this section, we prove {\red the existence of} {\blue {\magenta a solution
	in the EVI sense, by}
	 following the gradient flow approach introduced in \cite{AGS}}.

\bigskip

We will work almost always in $D(E)$ (i.e. $ E(u)<+\8$), which, as shown in Lemma \ref{basic properties of E} below, is contained in $W_{per_0}^{2,3}(I)$.
It is worthy noting that, as our energy requires $u_{xx}+a {  \ge } 0$ a.e., this non-negative condition
is preserved when taking the limit. Indeed, let $u_n\sse D(E)$ be a sequence with 
$\sup_n E(u_n)<+\8$, then $\sup_n \|u_n\|_{W_{per_0}^{2,3}(I)}<+\8$, 
hence,  up to a sub-sequence, $u_n$ converges strongly in $L^2_{per_0}(I)$ to some function $u\in L^2_{per_0}(I)$, satisfying $u_{xx}+a  { \ge } 0$ a.e.. 

{\red 
	Before proving the properties for the energy functional $E$, we recall the definition of $\la$-convexity from \cite{AGS} below.
\begin{defn}\label{lambda convexity}
Given a functional $\phi:   L^2_{per_0}(I) \rightarrow ( -\infty, +\infty ]$, we say $\phi$ is $\lambda$-convex along curves in the metric space $(L^2_{per_0}(I), \|\cdot\|_{L^2(I)})$ if
\begin{align*}
\phi((1-t)\gamma_0+t \gamma_1) \leq (1-t) \phi(\gamma_0) + t \phi (\gamma_1) - \frac{1}{2}\lambda t(1-t) \|\gamma_0- \gamma_1\|^2_{L^2(I)}  \qquad \forall t \in [0,1],
\end{align*}
for any $\gamma_{0},\gamma_1\in L_{per_0}^2(I)$.
\end{defn}
}

\begin{lem}
	The energy $E$ is 
	$2C$-convex with $C:=\sqrt{3}-2\ln 2>0$, 
	and lower semi-continuous with respect to the weak $L^2$-topology. Moreover, 
	{\magenta the} sub-levels of $E$ are compact in 	the  {\red strong $L^2$} topology.
	\label{basic properties of E}
\end{lem}

\begin{proof}
	{\em (1) Boundedness of $E$ from below.}
	
	 Since for any $u\notin D(E)$ we have $E(u)=+\8$,
	we need only to prove $E>-\8$ on its domain $D(E)$.
	First, for the second part of the energy $E(u)$ in \eqref{definition of E},  given $u\in L^3(I)$ such that $u_{xx}+a\ge 0$ a.e., we have 
	\begin{align}
           \int_I \Phi(u_{xx}+a)\d x& =\frac{1}{2}\|u_{xx}+a\|_{L^3(I)}^3
	+ \int_I (u_{xx}+a ) \ln (u_{xx}+a )\d x \notag\\
	&\ge
	\frac{1}{2}\|u_{xx}+a\|_{L^3(I)}^3 +
	|I| \cdot \inf_{\xi>0} \xi\ln\xi .  \label{tm-2}
           \end{align}
	Second, we turn to {\red estimating} the first  part of the energy $E(u)$ in \eqref{definition of E}. Denote 
	$$g(\xi):=\ln \left|\sin(\pi \xi)\right|\le 0.$$ {\red The first term in $E$ becomes
	\begin{align*}
		\frac{1}{|I|}\int_I 
		(u_{xx}+a) \bigg[\underbrace{ \int_I g(x-y) (u_{yy}+a)(y)  \d y}_{:=T(x)} \bigg]\d x.
	\end{align*}
	}
Since $g\le 0$, $u_{yy}+a\ge 0$, we have
\begin{align}
0\le -T(x)&=- \int_{\mathbb{R}}
g(x-y) (u_{yy}+a) \mathbf{1}_I(y) \d y=-\int_\R 
[g\cdot\mathbf{1}_{(x-1,x)}](x-y) [(u_{yy}+a) \cdot\mathbf{1}_I] (y) \d y\notag\\
&\le 
-\int_\R 
[g\cdot\mathbf{1}_{(-1,1)}](x-y) [(u_{yy}+a) \cdot\mathbf{1}_I] (y) \d y\notag\\
&=-\left\{ [g\cdot\mathbf{1}_{(-1,1)}] * [(u_{yy}+a) \cdot\mathbf{1}_I] \right\}(x),
\label{estimating T}
\end{align}
where we {\red used the fact that} $x\in (0,1)$ {\red implies} $(x-1,x)\sse(-1,1)$.
Therefore, by Young's inequality, {\red we can estimate the absolute value of the first term in $E$:}
\begin{align}
\bigg|\frac{1}{|I|}\int_I 
(u_{xx}+a)&\bigg[\int_I g(x-y)  (u_{yy}+a) \d y\bigg]\d x \bigg|\notag\\
&=\frac{1}{|I|}\int_I 
(u_{xx}+a)\bigg[\underbrace{-\int_I g(x-y)  (u_{yy}+a) \d y}_{=-T(x)\ge 0}\bigg]\d x\notag\\
&\overset{\eqref{estimating T}}{\le}
\frac{1}{|I|}\int_I 
\bigg|(u_{xx}+a)  \left[ (-g\cdot\mathbf{1}_{(-1,1)}) * \left( \mathbf{1}_I\cdot (u_{yy}+a) \right) \right]\bigg| \d x\notag\\
&\le \frac{1}{|I|} \|u_{xx}+a\|_{L^2(I)} \left\|[g\cdot\mathbf{1}_{(-1,1)}] * \left( \mathbf{1}_I\cdot(u_{yy}+a) \right) \right\|_{L^2(I)}\notag \\
&\le \frac{1}{|I|} \left\|g \right\|_{L^1(-1,1)}  \|u_{xx}+a\|_{L^2(I)}^2,
\label{estimating convolution term}
\end{align}
where
\begin{align}\label{2ln}
\left\|g \right\|_{L^1(-1,1)} &= 2\int_0^1 -\ln |\sin (\pi\xi)|\d\xi
=\frac{2}{\pi}\int_0^\pi -\ln |\sin ( w)|\d w = 2\ln 2.
\end{align}
Combining this with \eqref{tm-2}, we obtain
\begin{equation}
\begin{aligned}
 E(u) &\ge \frac{1}{2} \|u_{xx}+a\|_{L^3(I)}^3 + |I|\cdot \left(\inf_{\xi>0} \xi\ln \xi \right)
- \left\|g \right\|_{L^1(-1,1)} \|u_{xx}+a\|_{L^2(I)}^2\\
&= \frac{1}{2} \|u_{xx}+a\|_{L^3(I)}^3 -\frac{1}{e}
- (2\ln 2) \|u_{xx}+a\|_{L^2(I)}^2.
\end{aligned}
\label{boundedness from below of E}
\end{equation}
Thus $-\8 <\inf E \le E(u)<+\8$  implies that $u\in W^{2,3}_{per_2}(I)$. Hence 
 $D(E)\subset W_{per_0}^{2,3}(I)$.

{\red Moreover, since $\ln\xi \le \xi$ for any $\xi \ge 1$, we get $\ln (u_{xx}+a ) \le u_{xx}+a  $ 
	whenever $u_{xx}+a \ge 1$, {and}  
	\begin{align*}
	 \int_I (u_{xx}+a ) \ln (u_{xx}+a )\d x 
              &    = 	 \int_{ \{u_{xx}+a\ge 1  \} }  (u_{xx}+a ) \ln (u_{xx}+a )  \d  x  + 
	 	 \int_{\{u_{xx}+a< 1  \} }(u_{xx}+a ) \ln (u_{xx}+a )\d x    \\
	 	 & \le \|u_{xx}+a\|^2_{L^2(I)} + |I| \cdot \sup_{1>\xi>0} \xi\ln\xi.
	\end{align*}
This, together with the estimate \eqref{estimating convolution term}  for the 	first term in $E$, shows that
\begin{equation}\label{boundsE}
E(u) \leq \frac12 \|u_{xx}+a\|_{L^3(I)}^3 +\left(\frac{2\ln 2}{|I|} + 1 \right) \|u_{xx}+a\|_{L^2(I)}^2\leq \|u_{xx}+a\|_{L^3(I)}^3+c . 
\end{equation}
	
	 }
	
	\medskip
	
		{\em (2) $\lambda$-convexity in $L^2_{per_0}(I)$.} 
		
		{\magenta
		First, note that if in the $\la$-convexity inequality
		\begin{align*}
		E((1-t) u_0+t u_1) \leq (1-t) E(u_0) + t E(u_1) - \frac{1}{2}\lambda t(1-t) \|u_0- u_1\|^2_{L^2(I)}  
		\end{align*}
		we have either $E(u_0) =+\8 $ or $  E(u_1)=+\8$, then the inequality is trivial. Thus assume
		both terms are finite. This requires $ (u_i)_{xx}+a\ge 0 $ a.e., $i=1,2$, and
		hence $((1-t) u_0+t u_1)_{xx}+a\ge 0  $ a.e. too. Thus we can restrict our attention
		to functions $u$ such that $u_{xx}+a\ge 0$ a.e..
	}
	
		Note that 
	$   ( \Phi(\xi) -\sqrt{3}\xi^2 )^{\prime \prime}  = 3\xi +\xi^{-1}-2\sqrt{3} \ge 0$ {\red for $\xi>0$.} Hence 
\begin{equation}\label{convex1}
	u\mapsto\int_I [\Phi(u_{xx}+a) - \sqrt{3}(u_{xx}+a)^2] \d x\qquad \text{ is convex. }
\end{equation}	
	Rewrite the energy as
	\begin{equation}\label{energy_rew}
	\begin{aligned}
	E(u)& =\underbrace{ \int_I [\Phi(u_{xx}+a) - \sqrt{3}(u_{xx}+a)^2] \d x}_{\text{convex}} \\
	&\qquad
	+ \sqrt{3}\|u_{xx}+a\|_{L^2(I)}^2 
	+\frac{1}{|I|}\int_I 
	(u_{xx}+a) \bigg[\int_I g(x-y)(u_{yy}+a)\d y\bigg]\d x. 
	\end{aligned}
	\end{equation}
	Next we will prove that the sum of  the last two terms in $E(u)$ above is $\lambda$-convex.
	
	Given $u,v\in D(E)$, $t\in [0,1]$, notice on the one hand,
	\begin{align*}
	\int_I 
	 \int_I &g(x-y) [(1-t)(u_{yy}+a)+t(v_{yy}+a)]\cdot [(1-t)(u_{xx}+a)+t(v_{xx}+a)] \d y\d x\\
	 & = \int_I 
	 \int_Ig(x-y) \bigg[(1-t)(u_{xx}+a)(u_{yy}+a) + t(v_{xx}+a)(v_{yy}+a)\\
	&\quad -t(1-t) (u-v)_{xx} (u-v)_{yy}  \bigg]\d x\d y\\
	 & = (1-t)\int_I  \int_I g(x-y) (u_{xx}+a)(u_{yy}+a) \d x\d y
	 +t\int_I  \int_I g(x-y) (v_{xx}+a)(v_{yy}+a) \d x\d y\\
&\quad - t(1-t)\int_I  \int_Ig(x-y) (u-v)_{xx} (u-v)_{yy} \d x\d y.
	\end{align*}
	On the other hand, 
	\[ \|(1-t)(u_{xx}+a)+t(v_{xx}+a) \|_{L^2(I)}^2 =
	(1-t)\|u_{xx}+a\|_{L^2(I)}^2 + t\|v_{xx}+a \|_{L^2(I)}^2 -t(1-t)\|(u-v)_{xx} \|_{L^2(I)}^2 .\]
	Thus
	\begin{align}
	 &\sqrt{3}\|(1-t) (u_{xx}+a) +  t (v_{xx} +a )  \|_{L^2(I)}^2 
	+\frac{1}{|I|}\int_I 
	\int_I  g(x-y)  [(1-t)(u_{yy}+a)+t(v_{yy}+a)]\notag\\
	&\quad\cdot [(1-t)(u_{xx}+a)+t(v_{xx}+a)] \d y\d x
	\notag\\
	=& (1-t)\bigg[\sqrt{3}\|u_{xx}+a\|_{L^2(I)}^ 2 + \frac{1}{|I|}\int_I  \int_Ig(x-y) (u_{xx}+a)
	(u_{yy}+a) \d x\d y\bigg]\notag\\
	&\quad+ 
	t\bigg[\sqrt{3}\|v_{xx}+a \|_{L^2(I)}^2+ \frac{1}{|I|}\int_I  \int_Ig(x-y) (v_{xx}+a)
	(v_{yy}+a) \d x\d y \bigg] \notag\\
	&\quad - t(1-t) \bigg[ \sqrt{3}\|(u-v)_{xx} \|_{L^2(I)}^2
	+\frac{1}{|I|} \int_I  \int_Ig(x-y) (u-v)_{xx} (u-v)_{yy} \d x\d y\bigg].
	\label{convexity convolution part}
	\end{align}
	From \eqref{estimating convolution term}, we know that
	\begin{align}
		 \sqrt{3}\|(u-v)_{xx} \|_{L^2(I)}^2
	&+\frac{1}{|I|} \int_I  \int_Ig(x-y)(u-v)_{xx} (u-v)_{yy} \d x\d y
	\notag\\
	&\ge C\|(u-v)_{xx} \|_{L^2(I)}^2\geq C\|u-v\|_{L^2(I)}^2,\label{C definition}
		\end{align}
{\magenta		where 
		\[C:=\sqrt{3}-\|g\|_{L^1(-1,1)}=
		\sqrt{3}-2\ln 2>0.\]}
This, together with \eqref{convexity convolution part} implies that
\[u\mapsto \sqrt{3}\|u_{xx}+a\|_{L^2(I)}^2 
+\frac{1}{|I|}\int_I 
(u_{xx}+a) \bigg[\int_I g(x-y) (u_{yy}+a) \d y\bigg]\d x\]
is $\la$-convex in $L^2(I)$ with $\la=2C$.
Thus \eqref{convex1} follows, and  $E$ is also $2C$-convex in $L^2(I)$.

	\medskip
	
	{\em (3) Lower semi-continuity.} 
	
	Consider a sequence $u_n\to u$ weakly in $L^{2}(I)$.
	We need to show 
	\[\liminf_{n\to+\8}E(u_n) \ge E(u).\]
	Assume, without loss of generality, the $\liminf$ 
	is an actual limit, and that $\sup_n E(u_n)<+\8$, as otherwise the inequality is trivial. So we know $  ( u_n)_{xx}  +a\geq 0$   and $u_{xx}+a\geq 0$ a.e..
	
	{\red Boundedness of energy $E(u_n)$ implies, by \eqref{boundedness from below of E}, that 
	$(u_n)_{xx}+a$ is bounded in $L^3(I)$. Then we know 	 $   ( u_n)_{xx}    \to u_{xx}$ weakly in $L^3(I)$ and  $u_n\to u$ strongly in $H^{1}(I)$. 
	Therefore, 
	$$\|u_{xx}\|_{L^3(I)}^3\leq \liminf_{n\to +\8} \|   ( u_n)_{xx}    \|_{L^3 (I)}^3<+\8.$$
This, together with \eqref{boundsE}, implies $E(u)<+\8$.}


{\magenta We recall {\blue the previous} \eqref{energy_rew}.
For the last term,
we have
\begin{align*}
	\liminf_{n\to +\8}&
\frac{1}{|I|}\int_I 
(   (u_n)_{xx}+a) \bigg[\int_I g(x-y)(   (u_n)_{yy}+a)\d y\bigg]\d x\\
&=
\frac{1}{|I|}\int_I 
(u_{xx}+a) \bigg[\int_I g(x-y)(u_{yy}+a)\d y\bigg]\d x 
\end{align*}
	due to the dominated convergence theorem.
	The other term
	\[ \underbrace{ \int_I [\Phi(u_{xx}+a) - \sqrt{3}(u_{xx}+a)^2] \d x}_{\text{convex}}+
	 \sqrt{3}\|u_{xx}+a\|_{L^2(I)}^2   \]
	is convex and weakly lower semi-continuous. Thus 
	$$E(u)\leq \liminf_{n\to+\8} E(u_n),$$	
	and hence we conclude {$E$ is lower semi-continuous}  with respect to the weak $L^2$-topology.
}
%
%
	\medskip
	
		{\em (4) Compactness of sub-levels.} 
\tcr{ Consider a sequence $u_n $ with $E(u_n) \leq  c$. 
	{Boundedness of energy $E(u_n)$ implies, by \eqref{boundedness from below of E}, that 
		$(u_n)_{xx}+a$ is bounded in $L^3(I)$, thus there exits $u$ such that
		 $   ( u_n)_{xx}    \to u_{xx}$ weakly in $L^3(I)$, and  $u_n\to u$ strongly in $L^{2}(I)$. }
By the lower semi-continuity of $E$, 
\begin{eqnarray}
E(u ) \leq \liminf_{n\to+\8}E(u_n)  \leq c.
\end{eqnarray}
{\magenta Thus we complete the proof of this lemma.}
}
\end{proof}
%
%

\begin{proof}
	(of Theorem \ref{main theorem 1}) 
	{\red Notice in
	Lemma \ref{basic properties of E} we {show} that
	 all  hypotheses of  \cite[Theorem~4.0.4]{AGS}   are satisfied, with
	energy $E$, Hilbert space $L^{2}_{per_0}(I)$, and $\la=2C>0$. 
	
	Then by {the conclusions} (ii) and (iii) of \cite[Theorem~4.0.4]{AGS}, we know there exists a unique solution $u$ such that $u(t)\in D(E), \, t>0$  is a locally absolutely continuous curve with $\lim_{t\to 0^+} u(t) = u_0$ in $L^2(I)$ and 
	\begin{equation}\label{dt-es}
	\frac12 \frac{\d}{\d t}\|u(t)-v\|_{L^2}^2 + \frac12 \lambda \|u(t)-v\|_{L^2}^2 + E(u(t)) \leq E(v), \quad a.e.\, t>0, \forall v\in D(E).
	\end{equation}
Then combining it with the lower bound estimate for $E$ in \eqref{boundedness from below of E}, we conclude
$$u\in L^\8_{loc}(0,+\8; W^{2,3}_{per_0}(I)).$$

Now we turn to proving  $u_t\in L^\8(0,+\8; L^2(I))$.	
First, we know that $t\mapsto u(t)$ is locally Lipschitz in
$(0,+\8)$, i.e. for any $t_0>0$ there exists $L=L(t_0)>0$ such that
$$\|u(t_0+\vep)-u(t_0)\|_{L^2(I)}\le L(t_0)\vep \qquad \text{for all } {\vep\geq 0}.$$
Next, we {\magenta need to show that such {\blue $L(t_0)$ 
can be essentially taken independent of $t_0$}
}. For any $t_0\geq 0$,
from \eqref{dt-es} and $\lambda=2C>0$, we have
	\begin{equation}\label{dt-unif}
	\frac12 \frac{\d}{\d t}\|u(t)-u(t_0)\|_{L^2}^2  \leq E(u(t_0))-E(u(t)), \quad a.e.\, t>0.
	\end{equation}	
	Then by the conclusion (ii) of \cite[Theorem~4.0.4]{AGS}, and the lower bound estimate for $E$ in \eqref{boundedness from below of E}, we know
	\begin{equation}
	\frac{\d}{\d t}\|u(t)-u(t_0)\|_{L^2}^2 \leq E(u_0) +c_0<\8,
	\end{equation}
	where $c_0$ is an {\magenta uninfluential} constant.
	Thus the function $t\mapsto \|u(t_0)-u(t)\|_{L^2(I)}$
is globally Lipschitz with Lipschitz constant less than $E(u_0)+c_0$, which is independent of $t_0$.
From \cite[Theorem 1.17]{barbu2010nonlinear}, $u$ is differentiable a.e. in $[0,T]$ w.r.t $L^2(I)$, and belongs to $W^{1,\infty}([0,T];L^2(I))$.
Hence we know
$$\left\|\frac{u(t_0)-u(t_0+\vep)}{\vep}\right\|_{L^2(\om )}\le  E(u_0)+c_0.$$
Thus for a.e. $t$ we have
\begin{equation*}
\dfrac{u(t+\vep)-u(t)}\vep \in L^2(I),\qquad \left\|\dfrac{u(t+\vep)-u(t)}\vep\right\|_{L^2(\om)}\le  E(u_0)+c_0,
\end{equation*}
and the sequence of difference quotients $\dfrac{u(t+\vep)-u(t)}\vep$ is uniformly bounded in $L^2(I)$. Since $u$ is differentiable a.e. in $[0,T]$ and the derivative is unique, we can define 
$$\pd_t u(t):=\lim_{\vep\to 0}\dfrac{u(t+\vep)-u(t)}\vep,$$ and
consequently,
 \begin{equation}\label{ut925}
 \|\pd_t u\|_{L^\8(0,T;L^2(I))}\le E(u_0)+c_0 .
 \end{equation}
The proof is thus complete.	}
\end{proof}

\section{Higher regularity and globally positivity}\label{sec3}
{\red  In this section, {we concentrate on proving}  Theorem \ref{main theorem 2} for the existence and regularity of the strong solution to \eqref{u_eq}, and {for} the positive lower bound for $u_{xx}+a$. We will  first  calculate the sub-differential of $E$ for $u_{xx}+a>0$ in Section \ref{sec_subD}.
Assume $T_{\max}$ is the maximal time  (including the case
	$T_{\max}=+\8$) such that
\begin{equation}\label{first_lower}
u_{xx}(t)+a \geq \frac{c^*}{2} >0, \quad t\in[0,T_{\max}]
\end{equation}
for some positive constant $c^*>0.$
From the local-in-time smooth solution obtained in \cite{gao2017continuum}, we know if {the} initial {datum} $u_0$ satisfies $(u_0)_{xx}+a >\frac{c^*}{2}$, then $T_{\max}>0.$ In Section \ref{sec_lower}, we will give the key a priori estimate to show that indeed there is a uniform lower bound $c^*$ such that $u_{xx}(t)+a \geq c^*$ 
for all times $t$,
 and thus $T_{\max}=+\8$.
 {\magenta This significantly simplifies the sub-differential computations, since one of the key issues
 is the singularity given by the logarithmic term $\ln( u_{xx}(t)+a )$.  }
  Finally, we will prove Theorem \ref{main theorem 2} in Section \ref{sec_th2}.
}
\subsection{Sub-differential computations when $u_{xx}+a>0$}\label{sec_subD}

{\red In this  subsection, we  calculate the sub-differential of $E$ when $u_{xx}+a>0$.
The main result is:	
}
\begin{lem}\label{sub-differential single-valued}
	[sub-differential is single-valued] For any $u\in D(\pd E)$ such that $u_{xx}+a>0$, the sub-differential $\pd E(u)$ is single-valued and is given by
	\begin{equation}\label{subE}
	 \pd E(u) = {\magenta \{} [2\pi H(u_x)+\Phi'(u_{xx}+a)  ]_{xx} {\magenta \} }.
	\end{equation} 
\end{lem}

\begin{proof}
Step 1. We first prove 	\begin{equation}\label{sub_com}
		 [2\pi H(u_x)+\Phi'(u_{xx}+a)  ]_{xx} \in \pd E(u) . 
	\end{equation}
	Consider an arbitrary $u\in D(\pd E)\sse D(E)$, and let $\vph$ be a test function. 
	Without loss of generality, we can assume $u+\vep\vph\in D(E)$ too, because otherwise we would
	have
	\[ \lim_{\vep\to 0} \frac{\Phi_a(u_{xx}+\vep \vph_{xx}) - \Phi_a(u_{xx})}{\vep} =+\8,\]
	which immediately yields \eqref{sub_com}.
	
	We calculate the
	 elements of $\pd E(u) $ term by term.  First,  by the convexity of $\Phi_a$, we have
	\begin{equation}
	\label{good convex part}
	\vep\int_I \Phi_a'(u_{xx}) \vph_{xx} \le {\red \int_I [ }   \Phi_a(u_{xx}+\vep \vph_{xx}) - \Phi_a(u_{xx}) {\red ]},
	\end{equation}
  and then the term $[ \Phi_a'(u_{xx})]_{xx} $ belongs to the sub-differential of $\int_I \Phi_{\red a} (u_{xx}) \d x.$
	
	Next, we analyze the first part in the energy term:
	\begin{align}
	& \int_I \int_I  \ln |\sin( \pi(x-y) )|  (u_{yy}+\vep \vph_{yy}+a) (u_{xx} +\vep \vph_{xx} +a) \d y \d x
	 \notag\\
	& \qquad-
	 	 \int_I \int_I \ln |\sin( \pi(x-y) )|  (u_{yy}+a) (u_{xx} +a) \d y \d x\notag\\
	 	 =& 	\vep \int_I \int_I \ln |\sin( \pi(x-y) )|  [\vph_{yy}(u_{xx} +a) 
	 	 +\vph_{xx}(u_{yy} +a)]\d y \d x \label{vep order term} \\
	 	 &\qquad+\vep^2 \int_I \int_I \ln |\sin( \pi(x-y) )| \vph_{xx}\vph_{yy}\d y\d x.
	 	 \label{squared vep order term}
	\end{align}
Again, by writing as a convolution, we have
\begin{align*}
\int_I \int_I & \ln |\sin( \pi(x-y) )| \vph_{xx}\vph_{yy}\d y\d x  
= \int_\R  \bigg(\int_\R  \ln |\sin( \pi(x-y) )| \vph_{yy} \mathbf{1}_I(y) \d y\bigg) \vph_{xx} 
\mathbf{1}_I(x)\d x\\
& = 
 \int_\R \Big(  \ln |\sin( \pi\times \cdot) | * (  \mathbf{1}_I \vph'' )(x)  \Big) 
\vph_{xx}  \mathbf{1}_I(x)\d x.
\end{align*}
	Hence, noting that $x,y\in I$ implies $x-y\in (-1,1)$,
	\begin{align*}
	\bigg|\int_I \int_I & \ln |\sin( \pi(x-y) )| \vph_{xx}\vph_{yy}\d y\d x  \bigg|\\
	&\le \int_\R \Big|\Big(\big(  \mathbf{1}_{(-1,1)} \ln |\sin( \pi\times \cdot) | \big)
	* (  \mathbf{1}_I \vph'' )(x)  \Big) 
	\vph_{xx}  \mathbf{1}_I(x) \Big| \d x\\
	& \le
	\| \mathbf{1}_{(-1,1)} \ln |\sin( \pi\times \cdot) | * (  \mathbf{1}_I \vph'' ) \|_{L^2(\R)}
	\| \mathbf{1}_I \vph''\|_{L^2(\R)} \leq 2\ln 2
		\| \vph''\|_{L^2(I)}^2,
	\end{align*}
	where we   use \eqref{2ln}.
	Hence the term in \eqref{squared vep order term} is of order $O(\vep^2)$. 
	
	Now we turn our attention
	to the term in \eqref{vep order term}. Note first that, by a simple change of variable,
	\begin{align*}
	  &\int_I \int_I  \ln |\sin( \pi(x-y) )|  [\vph_{yy}(u_{xx} +a) 
+\vph_{xx}(u_{yy} +a)]\d y \d x\\
 =&
2 \int_I \bigg[\int_I \ln |\sin( \pi(x-y) )|  (u_{yy} +a) \d y \bigg]  \vph_{xx}\d x\\
=&
2 \int_I \bigg[\int_I \ln |\sin( \pi(x-y) )|  u_{yy}  \d y \bigg]  \vph_{xx}\d x
+ 2 a \underbrace{\bigg[\int_I \ln |\sin( \pi(x-y) )|   \d y \bigg] }_{=2\ln 2}
\underbrace{ \int_I\vph_{xx}\d x}_{=0}.
	\end{align*}
Note $\ln |\sin( \pi(x-y) )|$ has a $\ln$ like singularity at $x=y${\blue ;}
hence it belongs to $L^p(I)$ for all $p$, and $  u_{yy}$   belongs to $L^3(I)$. 
Thus   via integration by parts and the periodicity of $I$,
we have
\begin{align}
 \int_I \ln |\sin( \pi(x-y) )|  u_{yy}  \d y  &= -\int_I   u_{y} \frac{\pd}{\pd y} \ln |\sin( \pi(x-y) )|   \d y  \label{integrate by parts Hilbert}\\
& = \int_{-1/2}^{1/2}   u_{y}(x-y) \frac{\pd}{\pd y} \ln |\sin( \pi y )|   \d y \notag\\
&= -\pi \bigg[\int_{-1/2}^{0}   u_{y}(x-y)  \cot(y)  \d y
-\int_0^{1/2}  u_{y}(x-y) \cot(y)\d y\bigg] .
\notag
\end{align}
Note both the above integrals have a singularity at $y=0$, so we need more careful estimates for the last line.
Since
\[ \bigg| \int_I \ln |\sin( \pi y )| u_{yy}(x-y)  \d y  \bigg| <+\8, \]
we have
\[ \lim_{\vep\to 0} \int_{-\vep}^\vep \ln |\sin( \pi y )| u_{yy}(x-y)  \d y=0. \]
Therefore, we could rewrite \eqref{integrate by parts Hilbert} as
\begin{align}
 &\int_I \ln |\sin( \pi y )| u_{yy}(x-y)  \d y\notag\\
 =&
\lim_{\vep\to 0} \bigg[\int_{\vep}^{1/2} \ln |\sin( \pi y )| u_{yy}(x-y)  \d y
+\int_{-1/2}^{-\vep} \ln |\sin( \pi y )| u_{yy}(x-y)  \d y\bigg] \notag\\
=&\pi  \lim_{\vep\to 0} \bigg[
\int_{\vep}^{1/2} \cot( \pi y ) u_{y}(x-y)  \d y
-\int_{-1/2}^{-\vep} \cot( \pi y ) u_{y}(x-y)  \d y
\label{good part should give Hilbert}
\\
&\qquad -\ln|\sin(\vep\pi)| u_y(x-\vep)+\ln|\sin(\vep\pi)| u_y(x+\vep)
\bigg]. \label{boundary terms}
\end{align}
The limit in \eqref{good part should give Hilbert} exists, and it gives the Hilbert transform
term $H(u_y)$. For the other term \eqref{boundary terms}, we recall that
$u_{yy}\in L^3(I)${\blue ;} hence $u_y\in W^{1,3}(I) \sse C^{0,2/3}(I)$. That is,
there exists some constant $C_1>0$, independent of $x,y$, such that
\[ | u_y(x+\vep)-u_y(x-\vep) | \le C_1|2\vep|^{2/3} ,  \]
and \eqref{boundary terms} is now bounded through
\[ \lim_{\vep\to 0}  \Big| \ln|\sin(\vep\pi)| \big(u_y(x+\vep)-u_y(x-\vep)\big) \Big|
\le C_1 \lim_{\vep\to 0} \Big||2\vep|^{2/3} \ln|\sin(\vep\pi)| \Big| =0.\]
Therefore we have
\begin{align}
	& \lim_{\vep \to 0}\frac{1}{\vep}\Big[\int_I \int_I  \ln |\sin( \pi(x-y) )|  (u_{yy}+\vep \vph_{yy}+a) (u_{xx} +\vep \vph_{xx} +a) \d y \d x
	 \notag\\
	& \qquad-
	 	 \int_I \int_I \ln |\sin( \pi(x-y) )|  (u_{yy}+a) (u_{xx} +a) \d y \d x\Big]\notag\\
	 	 = & 	 \int_I 2\pi H(u_x) \varphi_{xx} \d x.
	\end{align}
This, together with the term $[\Phi_a'(u_{xx})]_{xx}$ in \eqref{good convex part}, concludes the Step 1.

Step 2.  We show that the sub-differential $\pd E(u)$ is single-valued. 
   Assume there exists another element $\eta\in \pd E(u) $. To prove that $Au:= [2\pi H(u_x)+\Phi'(u_{xx}+a)  ]_{xx}=\eta$ as elements of $[W_{per_0}^{2,3}(I)]^*$,
	we just need to show that
	\begin{equation}
		 \lg Au ,\vph\rg_{[W_{per_0}^{2,3}(I)]^*,W_{per_0}^{2,3}(I)} = \lg \eta ,\vph\rg_{[W_{per_0}^{2,3}(I)]^*,W_{per_0}^{2,3}(I)}
		 \label{sub-differential to show}
	\end{equation}
	for all test functions $\vph$ belonging to a suitable dense set $Z(u)\sse W_{per_0}^{2,3}(I)$, which will be constructed below.
	Here $\lg  ,\rg_{[W_{per_0}^{2,3}(I)]^*,W_{per_0}^{2,3}(I)} $ denotes the duality pairing between
   $[W_{per_0}^{2,3}(I)]^*$ and $W_{per_0}^{2,3}(I) $, induced through the embedding chain $W_{per_0}^{2,3}(I)
\hookrightarrow L_{per_0}^2(I)\hookrightarrow [W_{per_0}^{2,3}(I)]^*$. By the  definition of sub-differential, we have
\begin{align*}
\lim_{\vep\to 0} \frac{E(u+\vep \vph)-E(u)}{\vep} & \ge \lg Au ,\vph\rg_{[W_{per_0}^{2,3}(I)]^*,W_{per_0}^{2,3}(I)},\\
\lim_{\vep\to 0} \frac{E(u+\vep \vph)-E(u)}{\vep} & \ge \lg \eta ,\vph\rg_{[W_{per_0}^{2,3}(I)]^*,W_{per_0}^{2,3}(I)},\\
\lim_{\vep\to 0} \frac{E(u-\vep \vph)-E(u)}{\vep} & \ge \lg Au ,-\vph\rg_{[W_{per_0}^{2,3}(I)]^*,W_{per_0}^{2,3}(I)},\\
\lim_{\vep\to 0} \frac{E(u-\vep \vph)-E(u)}{\vep} & \ge \lg \eta ,-\vph\rg_{[W_{per_0}^{2,3}(I)]^*,W_{per_0}^{2,3}(I)}.
\end{align*}
Therefore, if both the left hand side terms
\[\lim_{\vep\to 0} \frac{E(u\pm\vep \vph)-E(u)}{\vep}\]
are finite, then we can infer
\begin{align*}
\lim_{\vep\to 0} \frac{E(u)-E(u-\vep \vph)}{\vep} &\le  
 \lg Au ,\vph\rg_{[W_{per_0}^{2,3}(I)]^*,W_{per_0}^{2,3}(I)}\le
\lim_{\vep\to 0} \frac{E(u+\vep \vph)-E(u)}{\vep},\\
\lim_{\vep\to 0} \frac{E(u)-E(u-\vep \vph)}{\vep} &\le  
 \lg \eta ,\vph\rg_{[W_{per_0}^{2,3}(I)]^*,W_{per_0}^{2,3}(I)} \le
\lim_{\vep\to 0} \frac{E(u+\vep \vph)-E(u)}{\vep}.
\end{align*}
So we need to carefully choose $\vph$ such that 
\begin{equation}\label{Gd}
\lim_{\vep\to 0} \frac{E(u+\vep \vph)-E(u)}{\vep}=\lim_{\vep\to 0} \frac{E(u)-E(u-\vep \vph)}{\vep}
\end{equation}
  both exist. 

To prove the Gateaux derivative in \eqref{Gd} exists,
the only term in $E(u\pm\vep \vph)$ that might create issues is 
$\int_I\Phi(u_{xx}+a\pm \vep\vph_{xx})\d x $ since we need $u_{xx}+a\pm \vep\vph_{xx}> 0$ a.e. to ensure \eqref{Gd} is finite.
Let
\begin{equation}\label{z-set}
Z_n(u):= \big\{\vph\in  W_{per_0}^{2,\8}(I): \vph_{xx}=0 \text{ on } \{u_{xx}+a<1/n\} \big\}
,\qquad Z(u):=\bigcup_{n\ge 1}Z_n(u).
\end{equation}
Therefore, for any $\vph \in Z(u)$, there exists ${\red Z_n}$ such that ${\red \vph_{xx}}=0$ therein. 
Then, by construction, for any $\vep <\frac{1}{n \|\vph_{xx}\|_{L^\8(I)}}$, we have $u_{xx}+a\pm \vep\vph> 0$ a.e..
It remains to check that $Z(u)$ is  dense in $W_{per_0}^{2,3}(I)$, i.e. for any
$\psi\in W_{per_0}^{2,3}(I)$ there exists a sequence $\psi_n\sse Z(u)$ such that $\psi_n\to \psi$ strongly in 
$W_{per_0}^{2,3}(I)$.
This is done in the Lemma \ref{single-valued dense} below. Therefore we conclude $\eta=Au$ in $[W_{per_0}^{2,3}(I)]^*$ and thus $\pd E$ is single-valued.
\end{proof}

{\magenta
For brevity, even though the sub-differential $\pd E(u)$ is a set,
we will simply write
\[ \pd E(u) =  [2\pi H(u_x)+\Phi'(u_{xx}+a)  ]_{xx} \]
instead.}

\begin{lem}
	The set $Z(u)$ constructed in \eqref{z-set} is dense in $W_{per_0}^{2,3}(I)$, i.e., for any
	$v\in W_{per_0}^{2,3}(I)$ there exists a sequence $v_n\sse Z(u)$ such that $v_n\to v$ strongly in 
	$W_{per_0}^{2,3}(I)$
	\label{single-valued dense}
\end{lem}

\begin{proof}
	Let $v\in W_{per_0}^{2,3}(I)$ be given, and we need to approximate $v$ with a sequence $v_n\sse Z(u)$. To this aim,
	we first approximate $v_{xx}$, and then take anti-derivatives. Let
	\[ w_n:= \min \left\{  v_{xx}\mathbf{1}_{\{u_{xx}+a\ge 1/n\}  } ,\,n \right\},
	 \]
	 which, intuitively, plays the role of $(v_{n})_{xx}$. That is, $w_n$ is constructed by first setting everything to zero
	 on $\{u_{xx}+a< 1/n\}$, and then   taking the truncation from above (at height $n$). Then, define
	 \begin{align*}
	 	  z_n(x):=\int_0^x w_n(s) \d s - \bar{w}_n ,\quad \bar{w}_n:=\frac{1}{|I|}\int_0^1 w_n(s) \d s,\quad 
	 	  v_n(x):=\int_0^x z_n(s) \d s. 	 	
	 \end{align*}
	 By construction,
	 	 $(v_{n})_{xx}=w_n$; hence $v_n\in Z(u)$ for any $n$. Moreover, since $z_n= (v_{n}) _{x} $ and $v_n$ have zero average,  by Poincar\'e's
	 	 inequality, 
	 	 we know $\| v_n-v \|_{L^3(I)} $ and $\| v_{nx}-v_x \|_{L^3(I)} $ are controlled by $\| (v_{n})_{xx}  -v_{xx} \|_{L^3(I)} $. By construction,
	 	 \[\| (v_{n})_{xx} -v_{xx} \|_{L^3(I)}^3 \le \int_{ \{u_{xx}+a< 1/n\}  } |v_{xx}|^3\d x+\int_{ \{v_{xx} \ge n\}  } |v_{xx}|^3\d x \to 0,\]
	 	 since the Lebesgue measures of both $\{u_{xx}+a< 1/n\}$ and $ \{v_{xx} \ge n\}  $ go to zero as $n\to+\8$. Thus we have shown that
	 	 $v_n\to v$ strongly in $W_{per_0}^{2,3}(I)$.
\end{proof}

{
\subsection{{\blue The a priori estimate}}
\label{sec_lower}
 In this subsection,  we  show the key a priori estimate
{\magenta which provides the existence of a uniform lower bound
 $c^*>0$, defined in \eqref{lowerB49} below, such that
the solution satisfies the global-in-time positivity property 
$$u_{xx}(t)+c\ge c^*>0,\qquad \fal t.$$
In other words, if the {\blue initial datum is} uniformly bounded away from zero, so is the solution at all
positive times.}

{ 
Let $u$  be a solution of
\[ u_t = - \frac{\dt E}{\dt u } = -[2\pi H(u_x) +\Phi'(u_{xx}+a)  ]_{xx}. \]
 satisfying \eqref{first_lower} for $t\in[0,T_{\max}]$.
Note that
\begin{align}
\frac{\d E}{\d t} = \int_I u_t \frac{\dt E}{\dt u }\d t = -
\int_I \bigg| \frac{\dt E}{\dt u }\bigg|^2\d t
\le 0, 
\label{energy is decreasing}
\end{align}
and
\begin{align*}
E(u_0) - \inf E & \ge
-\int_0^{+\8} \frac{\d E}{\d t}\d t =
\int_0^{+\8} \bigg\|\frac{\dt E}{\dt u }\bigg\|_{L^2(I)}^2 \d t\\
&  =\int_0^{+\8} \| [2\pi H(u_x(t)  ) +\Phi'(u_{xx}(t)+a)  ]_{xx}\|^2_{L^2(I)}  \d t\\
&\ge
C_I^{-1}\int_0^{+\8} \| [2\pi H(u_x(t)  ) +\Phi'(u_{xx}(t)+a)]_x  \|^2_{L^2(I)}  \d t\\
&=C_I^{-1}\int_0^{+\8} \Big\| 2\pi H(u_{xx}(t)  ) +\big[\ln(u_{xx}(t)+a) + \frac{3}{2} (u_{xx}(t)+a)^2
\big]_x \Big\|^2_{L^2(I)}  \d t,
\end{align*}
where $C_I$ is the Poincar\'e constant of $I$. 

\medskip

We show that the Hilbert transform term is controlled by the singular one. On the one hand, 
\begin{equation}\label{ln_ind}
\begin{aligned}
&\Big\| \big[\ln(u_{xx}(t)+a) + \frac{3}{2} (u_{xx}(t)+a)^2
\big]_x \Big\|^2_{L^2(I)}  
\\
=& \| [\ln(u_{xx}(t)+a)]_x \|^2_{L^2(I)}  
+ \frac{9}{4}\| [ (u_{xx}(t)+a)^2]_x\|^2_{L^2(I)}  
+3\int_I  [(u_{xx}(t)+a)^2]_x  [ \ln(u_{xx}(t)+a) ]_x\d x\\
=&
\| [\ln(u_{xx}(t)+a)]_x \|^2_{L^2(I)}  
+ \frac{9}{4}\| [ (u_{xx}(t)+a)^2]_x\|^2_{L^2(I)}  
+6\|  u_{xxx} \|^2_{L^2(I)}  .
\end{aligned}
\end{equation}
On the other hand, from \cite[Proposition~9.1.9]{butzer1971fourier},
\begin{equation}\label{tmc1}
4\pi^2 \|H(u_{xx}(t)  )\|^2_{L^2(I)}   = 4\pi^2 ( \|u_{xx}(t) +a \|_{L^2(I)}^2
-2a\|u_{xx}(t) +a \|_{L^1(I)}
+a^2
).
\end{equation}
By the Poincar\'e inequality,
\begin{align}\label{tmc2}
\| [ (u_{xx}(t)+a)^2]_x\|^2_{L^2(I)}  
&\ge C_I^{-1}\|  (u_{xx}(t)+a)^2\|^2_{L^2(I)}  =
C_I^{-1}\|  u_{xx}(t)+a\|^4_{L^4(I)}.
\end{align}
Combining \eqref{tmc1} and \eqref{tmc2}, there exists a computable constant
$C_0$ such that
\begin{equation}\label{controlH0}
 4\pi^2\| H(u_{xx}(t)  )\|^2_{L^2(I)}  \le \frac{1}{4} \| [ (u_{xx}(t)+a)^2]_x\|^2_{L^2(I)}  
\end{equation}
whenever $4\pi^2 \| H(u_{xx}(t)  )\|^2_{L^2(I)} \ge C_0$.

Thus
one of the following cases must hold:
\begin{enumerate}[(I).]
	\item The quantity $ 4\pi^2	\| H(u_{xx}(t)) \|_{L^2(I)}^2  \le C_0$. In this case
	\begin{align}
	&\Big\| 2\pi H(u_{xx}(t)  ) +\big[\ln(u_{xx}(t)+a) + \frac{3}{2} (u_{xx}(t)+a)^2
	\big]_x \Big\|^2_{L^2(I)} \notag\\
	 \ge&
	4\pi^2\| H(u_{xx}(t)  )\|^2_{L^2(I)}+
	\Big\|
	\big[\ln(u_{xx}(t)+a) + \frac{3}{2} (u_{xx}(t)+a)^2
	\big]_x \Big\|_{L^2(I)}  \notag\\
	&\qquad
	\times\Big[		\Big\|
	\big[\ln(u_{xx}(t)+a) + \frac{3}{2} (u_{xx}(t)+a)^2
	\big]_x \Big\|_{L^2(I)} -C_0  \Big].\label{control on the right hand term}
	\end{align}
	So the following dichotomy holds:
	\begin{enumerate}
		\item either
		\[  	\Big\|
		\big[\ln(u_{xx}(t)+a) + \frac{3}{2} (u_{xx}(t)+a)^2
		\big]_x \Big\|_{L^2(I)} \le 2C_0, \]
		 in which case we get a direct upper bound for $\Big\|
		\big[\ln(u_{xx}(t)+a) + \frac{3}{2} (u_{xx}(t)+a)^2
		\big]_x \Big\|_{L^2(I)}$;
		
		\item or
		\[  	\Big\|
		\big[\ln(u_{xx}(t)+a) + \frac{3}{2} (u_{xx}(t)+a)^2
		\big]_x \Big\|_{L^2(I)} \ge 2C_0, \]
	i.e.,
		 the last term in \eqref{control on the right hand term} satisfies
		\begin{align*}
			\Big\|
		\big[\ln(u_{xx}(t)+a) &+ \frac{3}{2} (u_{xx}(t)+a)^2
		\big]_x \Big\|_{L^2(I)} -C_0   \\
		&\ge
		\frac{1}{2}			\Big\|
		\big[\ln(u_{xx}(t)+a) + \frac{3}{2} (u_{xx}(t)+a)^2
		\big]_x \Big\|_{L^2(I)},
		\end{align*}
		so \eqref{control on the right hand term} gives
		\begin{align*}
		\Big\| 2\pi H(u_{xx}(t)  ) &+\big[\ln(u_{xx}(t)+a) + \frac{3}{2} (u_{xx}(t)+a)^2
		\big]_x \Big\|^2_{L^2(I)} \\
		& \ge \frac{1}{2}
		\Big\|
		\big[\ln(u_{xx}(t)+a) + \frac{3}{2} (u_{xx}(t)+a)^2
		\big]_x \Big\|_{L^2(I)}^{2} .
		\end{align*}
	\end{enumerate}

	\item Alternatively, if $4\pi^2	\| H(u_{xx}(t)) \|_{L^2(I)}^2  \ge C_0$, then from \eqref{ln_ind} and \eqref{controlH0}, we have the control
	\begin{align*}
	4\pi^2\| H(u_{xx}(t)  )\|^2_{L^2(I)} & \le\frac{1}{4}  \| [ (u_{xx}(t)+a)^2]_x\|^2_{L^2(I)} \\
	&
	\le \frac{1}{9}\Big\| \big[\ln(u_{xx}(t)+a) + \frac{3}{2} (u_{xx}(t)+a)^2
	\big]_x \Big\|^2_{L^2(I)}  ,
	\end{align*}
	which gives
	\begin{align*}
	\Big\| 2\pi H(u_{xx}(t)  ) &+\big[\ln(u_{xx}(t)+a) + \frac{3}{2} (u_{xx}(t)+a)^2
	\big]_x \Big\|^2_{L^2(I)}  \\
	&\ge	4\pi^2\| H(u_{xx}(t)  )\|^2_{L^2(I)}+
	\Big\| \big[\ln(u_{xx}(t)+a) + \frac{3}{2} (u_{xx}(t)+a)^2
	\big]_x \Big\|^2_{L^2(I)}\\
	&\qquad- 2\|2\pi  H(u_{xx}(t)  )\|_{L^2(I)} 
	\Big\| \big[\ln(u_{xx}(t)+a) + \frac{3}{2} (u_{xx}(t)+a)^2
	\big]_x \Big\|_{L^2(I)}\\
	&\ge
	4\pi^2\| H(u_{xx}(t)  )\|^2_{L^2(I)}+\frac{1}{3}
	\Big\| \big[\ln(u_{xx}(t)+a) + \frac{3}{2} (u_{xx}(t)+a)^2
	\big]_x \Big\|^2_{L^2(I)}.
	\end{align*}
	
\end{enumerate}
Combining all the above cases, we have 
\begin{align}
\Big\|
\big[\ln(u_{xx}(t)+a) \big]_x \Big\|_{L^2(I)}^2 & \overset{\eqref{ln_ind}}{\le}
\Big\|
\big[\ln(u_{xx}(t)+a) + \frac{3}{2} (u_{xx}(t)+a)^2
\big]_x \Big\|_{L^2(I)}^2 \notag\\
&\le \max\bigg\{
4C_0^2,  3 \Big\| 2\pi H(u_{xx}(t)  ) +\big[\ln(u_{xx}(t)+a) + \frac{3}{2} (u_{xx}(t)+a)^2
\big]_x \Big\|_{L^2(I)}^2 \bigg\}.
\label{control on ln}
\end{align}
Since by Lemma \ref{basic properties of E}, energy $E$ is $\la$-convex, with $\la=2C=2\sqrt{3}-4\ln 2>0$, it is well known
(see e.g. \cite[Theorem~2.4.15]{AGS})
that
\[ t\mapsto e^{\la t} \bigg\|\frac{\dt E(u(t))}{\dt u }\bigg\|_{L^2(I)}
= e^{\la t} \Big\| 2\pi H(u_{xx}(t)  ) +\big[\ln(u_{xx}(t)+a) + \frac{3}{2} (u_{xx}(t)+a)^2
\big]_x \Big\|_{L^2(I)}  \]
is nonincreasing.
Therefore, by the assumption $u^0_{xx}+a>0$ and denote 
\begin{align}
H_0:=\|\pd E(u^0)\|_{L^2(I)}= \Big\| 2\pi H(u_{xx}^0  ) +\big[\ln(u_{xx}^0+a) + \frac{3}{2} (u_{xx}^0+a)^2
\big]_x \Big\|_{L^2(I)}  <+\8, 
\label{uniform sub-gradient initial data}
\end{align}
we have
\[  \Big\| 2\pi H(u_{xx}(\cdot)  ) +\big[\ln(u_{xx}(\cdot)+a) + \frac{3}{2} (u_{xx}(\cdot)+a)^2
\big]_x \Big\|_{  L^\8(0,+\8;  L^2(I) }   \le H_0   .\]
 Combining this with \eqref{control on ln} finally gives
\[
\Big\|\ln(u_{xx}(\cdot)+a)\Big\|_{ L^\8( 0,+\8;  L^\8(I))}\le C_{\8,2}
\Big\|
\big[\ln(u_{xx}(\cdot)+a) \big]_x \Big\|_{ L^\8( 0,+\8;  L^2(I))} \le
C_{\8,2}
\max\{  2C_0, 3H_0 \},  \]
hence a uniform bound 
\begin{equation}\label{lowerB49}
c^*:=e^{-C_{\8,2} \max\{2C_0, 3H_0\}} 
\end{equation}
of $u_{xx}(\cdot)+a$ away from zero.}}

\subsection{Proof of higher regularity and Theorem \ref{main theorem 2}}\label{sec_th2}
Based on the calculations for the sub-differential,
 and the key a prior estimates {\magenta from the} previous subsections,  now we are in the position to prove higher order {\red regularity results} and Theorem \ref{main theorem 2}.
\begin{proof}(of Theorem \ref{main theorem 2}, statements (i)-(iii)).
From \cite[Proposition 1.4.4]{AGS},
$$|\pd E|(u(\cdot, t))= \min\{\|\xi\|_{L^2(I)}; \xi\in \pd E(u(\cdot, t)) \}.$$ From Lemma \ref{sub-differential single-valued}, we know $\pd E(u)$ is single-valued and is given by \eqref{subE}, which is {statement} (i) of Theorem \ref{main theorem 2}.
Thus statements (ii)-(iii) of Theorem \ref{main theorem 2}  follow directly from \cite[Theorem~2.4.15]{AGS} since
Lemma \ref{basic properties of E}
	{shows} that all its hypotheses 
	are satisfied. 
\end{proof}

\begin{proof}
	(of Theorem \ref{main theorem 2}, {statements} (iv) and (v)) 
	Let $v:=u_{xx}+a$.
	By  statement (iii)  of Theorem \ref{main theorem 2}, 
	the map 
	\[t\mapsto \left\| \frac{\dt E}{\dt u}(\cdot,t) \right\|_{L^2(I)}\exp(2C t)\]
	is nonincreasing and right continuous. Since $C>0$, this implies that 
	$t\mapsto \left\| \frac{\dt E}{\dt u}(\cdot,t) \right\|_{L^2(I)}$
	decreases exponentially in $t$.
	{\red Thus $u$ satisfies, for any $t\geq 0$,
	\[ -\frac{\dt E}{\dt u}=
	 [2\pi H(u_x)+\Phi'(v)  ]_{xx}=\Big[2\pi H(u_x)+\ln(u_{xx}+a) +\frac{3}{2}(u_{xx}+a)^2  \Big]_{xx}\]
	 is uniformly bounded in $ L^2( 0,+\8 ; L^{2}( I)) \cap L^\8( 0,+\8 ; L^{2}( I)),$
	which implies
	\begin{equation}
	[2\pi H(u_x)+\Phi'(v)  ]_{x}\in L^\8(0,+\8 ; H^{1}( I)),\qquad
		2\pi H(u_{x})+\Phi'(v)  \in L^\8(0,+\8 ; H^{2}( I)),
		\label{regularity sum proof of math thm 2}
	\end{equation}
where $v:= u_{xx}+a$ is a shorthand notation.}
	Using 
	\[u_{xx}+a \in L^\8( 0,+\8 ; L^{3}( I)) \Lra u_{xx},\ H(u_{xx})  \in L^\8( 0,+\8 ; L^{3}( I)),\]
	 we get
	\[ [\Phi'(v)  ]_{x}=
	[\ln v]_x +\frac{3}{2}[ v^2]_x \in  L^\8(0,+\8 ; L^{2}( I)) 
	\cap L^2( 0,+\8 ; L^{2}( I)). \]
	Then by \eqref{ln_ind},
	\begin{equation} \label{tm_1}
	\begin{aligned}
	+\8&>\int_0^{+\8}
	\left\| [\ln v(\cdot,t)]_x +\frac{3}{2}[v(\cdot,t)^2]_x  \right\|_{L^{2}( I)}^2 \d t\\
	& =\int_0^{+\8} \bigg\{\|\ln v(\cdot,t)]_x\|_{L^2(I)}^2 +\frac{9}{4}\|[v(\cdot,t)^2]_x\|^2 
	+6\|   u_{xxx}(\cdot,t)^2\|^2_{L^2(I)}\bigg\}  \d t;
	\end{aligned}
	\end{equation}
	hence
	\[[\ln v(\cdot,t)]_x,\  [v(\cdot,t)^2]_x,\  u_{xxx} \in L^2( 0,+\8 ; L^{2}( I)).\]
	Now that we have
	$u_{xxx}$ and $H(u_{xxx})\in L^2( 0,+\8 ; L^{2}( I))$, 
	we can use \eqref{regularity sum proof of math thm 2} to infer 
	\begin{equation*}
	[\Phi'(v)]_{x}=
	  [\ln v]_x +\frac{3}{2}[v^2]_x    \in L^2( 0,+\8 ; H^{1}( I))
	\end{equation*}
 and \eqref{more regularity 1} follows. Then, using the embedding
	$H^1(I)\hookrightarrow C^0(I)$, 
	\[  L^2( 0,+\8 ; C^0( I))\ni \left|[\Phi'(v)]_{x}\right|=
	|v_x(3v+v^{-1})| \ge  2\sqrt{3}|v_x|,\]
	which implies
\begin{equation}\label{tm_2}
	 v_x= u_{xxx}\in  L^2( 0,+\8 ; C^0( I)).
\end{equation}	
Similarly, since \eqref{tm_1}-\eqref{tm_2} also hold for any $t\geq 0$ uniformly, we conclude
	\eqref{more regularity 2} and \eqref{more regularity 3}.
	
The statement (v) {follows} from $\ln (u_{xx}+a)\in L^\8( 0,+\8 ; L^{\8}( I))$, i.e., 
$u_{xx}+a\geq c^*>0$ is bounded away from zero for all  $t>0$. Here the explicit positive lower bound $c^*$ is calculated in \eqref{lowerB49}. 

Next, by \eqref{more regularity 1}, for any $  t \geq 0 $,
\begin{align}
	+\8 &> \left\| [\ln v(\cdot,t)]_{xx} +\frac{3}{2}[v(\cdot,t)^2]_{xx} \right\|_{L^{2}( I)}^2 \notag\\
	&=\int_I \left[| [\ln v(\cdot,t)]_{xx}|^2+\frac{9}{4}|[v(\cdot,t)^2]_{xx}|^2\right]\d x
	+3\int_I  [\ln v(\cdot,t)]_{xx}\cdot [v(\cdot,t)^2]_{xx}\d x\\
	& =\int_I \left[| [\ln v(\cdot,t)]_{xx}|^2+\frac{9}{4}|[v(\cdot,t)^2]_{xx}|^2\right]\d x
	+6\int_I \left[ v_{xx}(\cdot,t)^2 - \frac{v_{x}(\cdot,t)^4}{v(\cdot,t)^2}\right]\d x.
	\label{proof of more regularity 4}
	\end{align}
 Note the only negative term is
	\[ -\int_I  \frac{v_{x}(\cdot,t)^4}{v(\cdot,t)^2}\d x,\]
	so we need to bound it from below. From \eqref{more regularity 3} and the uniform lower bound $u_{xx}+a\geq c^*>0$, we know
	\begin{align*}
	\int_I  \frac{v_{x}(\cdot,t)^4}{v(\cdot,t)^2}\d x \le 
c\| v_x(\cdot,t)\|_{C^0(I)}^2<+\8
	\end{align*}
	uniformly for $t\geq 0$.
	Hence \eqref{proof of more regularity 4} reads
	\begin{align*}
		+\8 &>\left\| [\ln v(\cdot,t)]_{xx} +\frac{3}{2}[v(\cdot,t)^2]_{xx} \right\|_{L^{2}( I)}^2 \\
		&\ge
		\left\{\int_I \left[| [\ln v(\cdot,t)]_{xx}|^2+\frac{9}{4}|[v(\cdot,t)^2]_{xx}|^2  +6v_{xx}(\cdot,t)^2\right]\d x\right\}
		\\
		&\quad-6 c\| v_x(\cdot,t)\|_{C^0(I)}^2,
	\end{align*}
	uniformly for $t\geq 0$
	and thus \eqref{more regularity 4} follows. This completes the proof of statement (iv).

\end{proof}

\begin{proof}(of Theorem \ref{main theorem 2}, statement (vi))
Since $u$ is periodic with {\red regularity} \eqref{more regularity 1}-\eqref{more regularity 4}, the steady state $u^*$ satisfies
$$\frac{\delta E}{\delta u}= [2\pi H(u_x)+\ln(u_{xx}+a)+\frac32(u_{xx}+a)^2 ]_{xx}=0,$$
which implies 
\begin{equation}
2\pi H(u_x)+\ln(u_{xx}+a)+\frac32(u_{xx}+a)^2 \equiv const.
\end{equation}
This yields $u^*\equiv 0$ is a steady state. From Lemma \ref{basic properties of E}, we know $\frac{\delta E}{\delta u}$ is strictly monotone in $L^2$, which implies there is only one steady state $u^*\equiv 0$ such that $\frac{\delta E}{\delta u}=0$. Thus combining  \cite[Theorem 2.4.14]{AGS} and Lemma \ref{basic properties of E}, we conclude the exponentially decay of $u(t)$ to its unique equilibrium $u^*=0$, i.e., statement (vi). 
\end{proof}

\section*{Acknowledgments}
XYL and CW warmly thank the support of NSERC Discovery Grants. XYL acknowledges the support
of his Lakehead University internal funding.

\bibliographystyle{abbrv}
\bibliography{strong}

\end{document}